\documentclass[12pt]{amsart}
\usepackage{amsfonts,amssymb,amscd,amsmath,enumerate,color}
\def\NZQ{\mathbb}               % the font for N,Z,Q,R,C

\def\ZZ{{\NZQ Z}}
\def\RR{{\NZQ R}}

\def\frk{\mathfrak}               % font for "Fraktur"

\def\Phi{{\frk N}}
\def\ab{{\bold a}}
\def\bb{{\bold b}}
\def\eb{{\bold e}}

\def\xb{{\bold x}}
\def\yb{{\bold y}}

\def\opn#1#2{\def#1{\operatorname{#2}}} % to make operators
\opn\chara{char} 
\opn\length{\ell} 
\opn\pd{pd} 
\opn\rk{rk}
\opn\projdim{proj\,dim} 
\opn\injdim{inj\,dim} 
\opn\rank{rank}
\opn\depth{depth} 
\opn\grade{grade} 
\opn\height{height}
\opn\embdim{emb\,dim} 
\opn\codim{codim}

\opn\Tr{Tr} 
\opn\bigrank{big\,rank}
\opn\superheight{superheight}
\opn\lcm{lcm}
\opn\trdeg{tr\,deg}%\emph{
\opn\reg{reg} 
\opn\lreg{lreg} 
\opn\ini{in} 
\opn\lpd{lpd}
\opn\size{size}
\opn\mult{mult}
\opn\dist{dist}
\opn\cone{cone}
\opn\lex{lex}
\opn\rev{rev}
\opn\div{div} \opn\Div{Div} \opn\cl{cl} \opn\Cl{Cl}
\opn\Spec{Spec} \opn\Supp{Supp} \opn\supp{supp} \opn\Sing{Sing}
\opn\Ass{Ass} \opn\Min{Min}
\opn\Ann{Ann} \opn\Rad{Rad} \opn\Soc{Soc}
\opn\Syz{Syz} \opn\Im{Im} \opn\Ker{Ker} \opn\Coker{Coker}
\opn\Am{Am} \opn\Hom{Hom} \opn\Tor{Tor} \opn\Ext{Ext}
\opn\End{End} \opn\Aut{Aut} \opn\id{id} \opn\ini{in}

\opn\nat{nat}
\opn\pff{pf}%   \pf exists already
\opn\Pf{Pf} \opn\GL{GL} \opn\SL{SL} \opn\mod{mod} \opn\ord{ord}
\opn\Gin{Gin}
\opn\Hilb{Hilb}\opn\adeg{adeg}\opn\std{std}\opn\ip{infpt}
\opn\Pol{Pol}
\opn\sat{sat}
\opn\Var{Var}
\opn\Gen{Gen}

\opn\aff{aff} \opn\con{conv} \opn\relint{relint} \opn\st{st}
\opn\lk{lk} \opn\cn{cn} \opn\core{core} \opn\vol{vol}
\opn\link{link} \opn\star{star}
\opn\gr{gr}

\def\Hc{{\mathcal H}}

\def\Fc{{\mathcal F}}
\def\Oc{{\mathcal O}}
\def\Pc{{\mathcal P}}
\def\Qc{{\mathcal Q}}

\def\pot#1#2{#1[\kern-0.28ex[#2]\kern-0.28ex]}

\opn\dirlim{\underrightarrow{\lim}}
\opn\inivlim{\underleftarrow{\lim}}
\let\to=\rightarrow

\def\Implies{\ifmmode\Longrightarrow \else
        \unskip${}\Longrightarrow{}$\ignorespaces\fi}
\def\implies{\ifmmode\Rightarrow \else
        \unskip${}\Rightarrow{}$\ignorespaces\fi}
\def\iff{\ifmmode\Longleftrightarrow \else
        \unskip${}\Longleftrightarrow{}$\ignorespaces\fi}

\let\:=\colon
\newtheorem{Theorem}{Theorem}[section]
\newtheorem{Lemma}[Theorem]{Lemma}
\newtheorem{Corollary}[Theorem]{Corollary}

\newtheorem{Example}[Theorem]{Example}

\let\epsilon\varepsilon
\let\phi=\varphi
\let\kappa=\varkappa
\def\qed{\hfill{
\ifhmode\textqed\fi
      \ifmmode\ifinner\quad\qedsymbol\else\dispqed\fi\fi}
}
\def\textqed{\unskip\nobreak\penalty50
       \hskip2em\hbox{}\nobreak\hfil\qedsymbol
       \parfillskip=0pt \finalhyphendemerits=0}
\def\dispqed{\rlap{\qquad\qedsymbol}}

\opn\dis{dis}
\opn\height{height}
\opn\dist{dist}
\def\pnt{{\raise0.5mm\hbox{\large\bf.}}}

\opn\Lex{Lex}

\begin{document}
\title{
Reverse lexicographic squarefree initial ideals and Gorenstein Fano polytopes
}
\author{Hidefumi Ohsugi and Takayuki Hibi
}


\address{Hidefumi Ohsugi,
Department of Mathematical Sciences,
School of Science and Technology,
Kwansei Gakuin University,
Sanda, Hyogo, 669-1337, Japan} 
\email{ohsugi@kwansei.ac.jp}

\address{Takayuki Hibi,
Department of Pure and Applied Mathematics,
Graduate School of Information Science and Technology,
Osaka University,
Toyonaka, Osaka 560-0043, Japan}
\email{hibi@math.sci.osaka-u.ac.jp}

\maketitle

\begin{abstract}
Via the theory of reverse lexicographic squarefree initial ideals of 
toric ideals, 
we give a new class of Gorenstein Fano polytopes
(reflexive polytopes) arising from a pair of
stable set polytopes of perfect graphs. 
\end{abstract}

\section*{Introduction}
Recall that an {\em integral} convex polytope is a convex polytope 
all of whose vertices have integer coordinates.  
An integral convex polytope $\Pc \subset \RR^{d}$ of dimension $d$
is called a {\em Fano polytope} if the origin of $\RR^{d}$ 
is a unique integer point belonging to the interior of $\Pc$. 
We say that a Fano polytope $\Pc \subset \RR^{d}$ is {\em Gorenstein}
if the dual polytope $\Pc^{\vee}$ of $\Pc$ is again integral.
(A Gorenstein Fano polytope is often called a {\em reflexive} polytope
in the literature.)
Gorenstein Fano polytopes are related with mirror symmetry and 
studied in a lot of areas of mathematics.
See, e.g., \cite[\S 8.3]{CLS} and \cite{survey}.
It is known that there are only finitely many Gorenstein Fano polytopes up to
unimodular equivalence if the dimension is fixed.
Classification results are known for low dimensional cases \cite{three, four}.
On the other hand,
one of the most important problem is
to construct new classes of Gorenstein Fano polytopes.
In the case of Gorenstein Fano {\em simplices}, 
there are nice results on classifications and constructions.
See, e.g., \cite{r1, r2, r3} and their references.
In order to find classes of Gorenstein Fano polytopes of high dimension
which are not necessarily simplices, integral convex
polytopes arising from some combinatorial objects are studied in several papers.
For example, the following classes are known:
\begin{itemize}
\item
Gorenstein Fano polytopes arising from the order polytopes of graded posets
(Hibi \cite{Hibiorder}, revisited by Heged\"us--Kasprzyk \cite[Lemma 5.10]{HK});
\item
Gorenstein Fano polytopes arising from the Birkhoff polytopes
(appearing in many papers.
See, e.g., Stanley's book \cite[I.13]{Sbook} and  Athanasiadis \cite{ata});
\item
Gorenstein Fano polytopes arising from directed graphs
satisfying some conditions
(Higashitani \cite{Higashi});
\item
Centrally symmetric configurations
(Ohsugi--Hibi \cite{CS});
\item
The centrally symmetric
polytope $\Oc(P)^{\pm}$ of the order polytope $\Oc(P)$
of a finite poset $P$ 
(Hibi--Matsuda--Ohsugi--Shibata \cite{HMOS}).
\end{itemize}
In the present paper, via the theory of Gr\"obner bases,
we give a new class of 
Gorenstein Fano polytopes which is not necessarily a simplex.
For any pair of perfect graphs $G_1$ and $G_2$ (here, $G_1 = G_2$ is possible)
on $d$ vertices, 
we show that the convex hull of ${\mathcal Q}_{G_1} \cup - {\mathcal Q}_{G_2}$,
where ${\mathcal Q}_{G_i}$ is the stable set polytope of $G_i$,
is a Gorenstein Fano polytope
of dimension $d$.
%Our class contains $p(d) (p(d)+1)/2$ Gorenstein Fano polytopes of dimension $d$
%where $p(d)$ is 
%the number\footnote{See A052431 in
%``The On-Line Encyclopedia of Integer Sequences,'' at 
%{\tt http://oeis.org}
%} of all perfect simple graphs on $d$ vertices
%is large, .
Note that there are a lot of pairs of perfect simple graphs on $d$ vertices\footnote{
See A052431 in ``The On-Line Encyclopedia of Integer Sequences,'' at 
{\tt http://oeis.org}
} (Figure 1).
%\begin{center}
\begin{figure}[h]
\begin{tabular}{|c|c|c|c|c|c|c|c|}
\hline
number of vertices& 2 & 3 & 4 & 5 & 6 & 7&8\\
\hline
perfect graphs & 2 & 4 & 11 & 33 & 148 & 906 & 8887\\
\hline
pairs of perfect graphs & 3 & 10 & 66 & 561 & 11,026 & 410,871 & 39,493,828\\
\hline
\end{tabular}
\caption{
Number of perfect graphs / pairs of perfect graphs
} 
\end{figure}
%\end{center}
Any Gorenstein Fano polytope ${\mathcal P}$
 in our class is {\em terminal}, i.e., 
each integer point belonging to the boundary of ${\mathcal P}$ is a vertex of ${\mathcal P}$.
In particular, if both of two graphs are the complete 
(resp.~empty) graphs on $d$ vertices, then the Gorenstein Fano polytope
has $2d$ (resp.~$2^{d+1}-2$) vertices.
Thus, our class has enough size and variety
comparing with the existing classes above.

Let $\ZZ_{\geq 0}$ denote the set of nonnegative integers.
Let     
$A = [{\bf a}_{1}, \ldots, {\bf a}_{n}] \in \ZZ_{\geq 0}^{d \times n}$ 
and 
$B = [{\bf b}_{1}, \ldots, {\bf b}_{m}] \in \ZZ_{\geq 0}^{d \times m}$,
where each ${\bf a}_{i}$ and each ${\bf b}_{j}$ is 
a nonzero column vector belonging to $\ZZ_{\geq 0}^{d}$.
In Section $1$, after reviewing basic materials and notation on toric ideals,
we introduce the concept that $A$ and $B$ are {\em of harmony}.
Roughly speaking, Theorem \ref{squarefree} says that 
if $A$ and $B$ are of harmony and 
if the toric ideal of each of $A$ and $B$ possesses
a reverse lexicographic squarefree initial ideal which enjoys certain 
properties, then the toric ideal of 
$[{\bf 0}, -B, A] \in \ZZ^{d \times (n + m+1)}$ 
possesses a squarefree initial ideal 
with respect to a reverse lexicographic order
whose smallest variable corresponds to the column
${\bf 0} \in \ZZ^{d}$. 
Working with the same situation as in Theorem \ref{squarefree},
Corollary \ref{Fano} guarantees that
if the integral convex polytope $\Pc \subset \RR^{d}$ which is 
the convex hull of 
$\{ - {\bf b}_{1}, \ldots, - {\bf b}_{m}, 
{\bf a}_{1}, \ldots, {\bf a}_{n} \}$
is a Fano polytope with 
$\Pc \cap \ZZ^{d} = \{ {\bf 0}, 
- {\bf b}_{1}, \ldots, - {\bf b}_{m}, 
{\bf a}_{1}, \ldots, {\bf a}_{n} \}$ 
and if 
there is a $d \times d$ minor $A'$ of $[-B, A]$ with
$\det(A') = \pm1$, 
then $\Pc$ is Gorenstein. 

The topic of Section $2$ is the incidence matrix $A_{\Delta}$ of a simplicial complex
$\Delta$ on $[d] = \{ 1, \ldots, d \}$.
It follows that if $\Delta$ and $\Delta'$ are simplicial complexes on $[d]$,
then $A_{\Delta}$ and $A_{\Delta'}$ are of harmony.  
Following Theorem \ref{squarefree} it is reasonable
to study the problem when the toric ideal of $A_{\Delta}$ satisfies 
the required condition on initial ideals of Theorem \ref{squarefree}.  
Somewhat surprisingly,
Theorem \ref{compressed} says that  
$A_{\Delta}$ satisfies 
the required condition on initial ideals of Theorem \ref{squarefree}
if and only if $\Delta$ coincides with 
the set $S(G)$ of stable sets of a perfect graph $G$ on $[d]$.
A related topic on Gorenstein Fano polytopes arising from simplicial
complexes will be studied (Theorem \ref{GORFANO}).

\section{Reverse lexicographic squarefree initial ideals}
Let $K$ be a field and $K[{\bf t}, {\bf t}^{-1}, s] = 
K[t_{1}, t_{1}^{-1}, \ldots, t_{d}, t_{d}^{-1}, s]$
the Laurent polynomial ring in $d + 1$ variables over $K$.
Given an integer $d \times n$ matrix 
$A = [{\bf a}_{1}, \ldots, {\bf a}_{n}]$, where
${\bf a}_{j} = [a_{1j}, \ldots, a_{dj}]^{\top}$, 
the transpose of $[a_{1j}, \ldots, a_{dj}]$, is
the $j$th column of $A$, 
the {\em toric ring} of $A$ is the subalgebra $K[A]$ of
$K[{\bf t}, {\bf t}^{-1}, s]$ which is generated by the Laurent
polynomials ${\bf t}^{{\bf a}_{1}}s = 
t_{1}^{a_{11}} \cdots t_{d}^{a_{d1}}s, \ldots,
{\bf t}^{{\bf a}_{n}}s = 
t_{1}^{a_{1n}} \cdots t_{d}^{a_{dn}}s$.
Let $K[{\bf x}] = K[x_{1}, \ldots, x_{n}]$ denote the polynomial ring
in $n$ variables over $K$ and define the surjective ring homomorphism
$\pi : K[{\bf x}] \to K[A]$ by setting $\pi(x_{j}) = {\bf t}^{{\bf a}_{j}}s$
for $j = 1, \ldots, n$.  The {\em toric ideal} of $A$ is the kernel $I_{A}$
of $\pi$.  Every toric ideal is generated by binomials.  
(Recall that a polynomial $f \in K[{\bf x}]$ is a binomial if 
$f = u - v$,  
% \prod_{i=1}^{n} x_{i}^{a_{i}} - \prod_{j=1}^{n} x_{j}^{b_{j}}$,
where $u = \prod_{i=1}^{n} x_{i}^{a_{i}}$ and 
$v = \prod_{i=1}^{n} x_{i}^{b_{i}}$
are monomials with $\sum_{i=1}^{n} a_{i} = \sum_{i=1}^{n} b_{i}$.)
Let $<$ be a monomial order on $K[{\bf x}]$ and 
${\rm in}_{<}(I_{A})$ the initial ideal of $I_{A}$ with respect to $<$.
We say that ${\rm in}_{<}(I_{A})$ is {\em squarefree} if 
${\rm in}_{<}(I_{A})$ is generated by squarefree monomials.
We refer the reader to \cite[Chapters 1 and 5]{dojoEN} 
for the information about Gr\"obner bases and toric ideals. 

Let $\ZZ_{\geq 0}^{d}$ denote the set of integer column vectors
$[a_{1}, \ldots, a_{d}]^{\top}$ with each $a_{i} \geq 0$.
Given an integer vector ${\bf a} = [a_{1}, \ldots, a_{d}]^{\top} \in \ZZ^{d}$, 
let ${\bf a}^{(+)} = [a_{1}^{(+)}, \ldots, a_{d}^{(+)}]^{\top}, 
{\bf a}^{(-)} = [a_{1}^{(-)}, \ldots, a_{d}^{(-)}]^{\top} \in \ZZ_{\geq 0}^{d}$ 
where $a_{i}^{(+)} = \max \{0, a_{i} \}$ and  $a_{i}^{(-)} = \max\{0, -a_{i}\}$.
Note that ${\bf a} = {\bf a}^{(+)} - {\bf a}^{(-)}$ holds in general. 
%there exist unique column vectors ${\bf a}^{(+)}$ and ${\bf a}^{(-)}$ belonging to 
%$\ZZ_{\geq 0}^{d}$ with ${\bf a} = {\bf a}^{(+)} - {\bf a}^{(-)}$. 
Let $\ZZ_{\geq 0}^{d \times n}$ denote the set of $d \times n$ 
integer matrices $(a_{ij})_{1 \leq i \leq d \atop 1 \leq j \leq n}$
with each $a_{ij} \geq 0$.
Furthermore if no columns of $A \in \ZZ_{\geq 0}^{d \times n}$
is the zero vector
${\bf 0} = [0, \ldots, 0]^{\top} \in \ZZ^{d}$, then 
we introduce the $d \times (n + 1)$ integer matrix $A^{\sharp}$
which is obtained by adding the column  
${\bf 0} \in \ZZ^{d}$ to $A$.

Now, given $A \in \ZZ_{\geq 0}^{d \times n}$
and $B \in \ZZ_{\geq 0}^{d \times m}$, we say that $A$ and $B$ 
are {\em of harmony} if the following condition is satisfied:
Let ${\bf a}$ be a column of $A^{\sharp}$ 
and ${\bf b}$ that of $B^{\sharp}$.  
Let ${\bf c} = {\bf a} - {\bf b} \in \ZZ^{d}$. 
If ${\bf c} = {\bf c}^{(+)} - {\bf c}^{(-)}$,
then ${\bf c}^{(+)}$ is a column vector of $A^{\sharp}$
and ${\bf c}^{(-)}$ is a column vector of $B^{\sharp}$.

\begin{Theorem}
\label{squarefree}
Let $A = [{\bf a}_{1}, \ldots, {\bf a}_{n}] \in \ZZ_{\geq 0}^{d \times n}$ 
and 
$B = [{\bf b}_{1}, \ldots, {\bf b}_{m}]\in \ZZ_{\geq 0}^{d \times m}$,
where none of ${\bf a}_{i}$'s and ${\bf b}_{j}$'s is
${\bf 0} \in \ZZ^{d}$,  
be of harmony.
Let $K[z, {\bf x}] = K[z, x_{1}, \ldots, x_{n}]$
and $K[z, {\bf y}] = K[z, y_{1}, \ldots, y_{m}]$
be the polynomial rings over a field $K$. 
Suppose that 
%the toric ideals $I_{A^{\sharp}}$ of $A^{\sharp}$ and
%$I_{B^{\sharp}}$ of $B^{\sharp}$ possesses a squarefree initial ideal 
${\rm in}_{<_A}(I_{A^{\sharp}})
\subset K[z, {\bf x}]$ and 
${\rm in}_{<_B}(I_{B^{\sharp}}) \subset K[z, {\bf y}]$ are
squarefree with respect to reverse lexicographic orders 
$<_A$ on $K[z, {\bf x}]$
and $<_B$ on $K[z, {\bf y}]$ respectively
satisfying the conditions that
\begin{itemize}
\item
$x_i  <_A  x_j$ \,if\, $\pi(x_i)$ \,divides\, $\pi(x_j)$;
\item
$z  <_A  x_k$ for $1 \leq k \leq n$, where 
$z$ corresponds to
the column ${\bf 0} \in \ZZ^{d}$ of $A^{\sharp}$;
\item
$z  <_B  y_k$ for $1 \leq k \leq m$, where 
$z$ corresponds to
the column ${\bf 0} \in \ZZ^{d}$ of $B^{\sharp}$. 
\end{itemize}
Let $[-B, A]$ denote the $d \times (n + m)$
integer matrix 
\[
[- {\bf b}_{1}, \ldots, - {\bf b}_{m}, 
{\bf a}_{1}, \ldots, {\bf a}_{n}].
\]
Then the toric ideal $I_{[-B, A]^{\sharp}}$ of
$[-B, A]^{\sharp}$ possesses a squarefree initial ideal 
with respect to a reverse lexicographic order
whose smallest variable corresponds to the column
${\bf 0} \in \ZZ^{d}$ of $[-B, A]^{\sharp}$. 
\end{Theorem}

\begin{proof}
% Let $K[\xb, \yb, z]=K[x_1,\ldots,x_n,y_1,\ldots,y_m,z]$
% be the polynomial ring in $n+m+1$ variables over $K$.
Let $K[[-B, A]^{\sharp}] \subset K[{\bf t}, {\bf t}^{-1}, s]
= K[t_{1}, t_{1}^{-1}, \ldots, t_{d}, t_{d}^{-1}, s]$
be the toric ring of $[-B, A]^{\sharp}$
and $I_{[-B, A]^{\sharp}} \subset K[\xb, \yb, z]
= K[x_1, \ldots, x_n, y_1, \ldots, y_m, z]$ 
the toric ideal of $[-B, A]^{\sharp}$.
Recall that $I_{[-B, A]^{\sharp}}$ is the kernel of 
% the surjective ring homomorphism
$\pi : K[\xb, \yb, z] \to K[[-B, A]^{\sharp}]$ 
with
% defined by setting 
$\pi(z) = s$,
$\pi(x_i) = {\bf t}^{{\bf a}_{i}}s$
for $i = 1, \ldots, n$ and
$\pi(y_j) = {\bf t}^{ - {\bf b}_{j}}s$
for $j = 1, \ldots, m$.

Suppose that the reverse lexicographic orders
$<_A$ and $<_B$ are induced by the orderings
$z <_A x_n <_A \cdots <_A x_1$
and
$z <_B y_m <_B \cdots <_B y_1$.
Let $<_{\rm rev}$ be the reverse lexicographic order on 
$K[\xb, \yb, z]$ induced by the ordering
\[
z < x_n < \cdots < x_1 < y_m < \cdots < y_1.
\]

In general, if $\ab = [a_{1}, \ldots, a_{d}]^{\top} \in \ZZ_{\geq 0}^{d}$, 
then ${\rm supp}(\ab)$ is the set of those $1 \leq i \leq d$ 
with $a_{i} \neq 0$.  Now, we introduce the following   
\[
{\mathcal E}
=
\{ \, (i,j) \, : \, 1 \leq i \leq n, \, 1 \leq j \leq m, \, 
{\rm supp} (\ab_i) \cap {\rm supp} (\bb_j) \neq 
\emptyset \, \}.
\] 
Let ${\bf c} = \ab_i - \bb_j$
with $(i,j) \in {\mathcal E}$.
Then ${\bf c}^{(+)} \neq \ab_i$
and
${\bf c}^{(-)} \neq \bb_j$.
The hypothesis that $A$ and $B$ are of harmony guarantees that
${\bf c}^{(+)}$ is a column of $A^{\sharp}$
and ${\bf c}^{(-)}$ is a column of $B^{\sharp}$.
It follows that $f = x_i y_j -u$ ($\neq 0$)
belongs to $I_{[-B, A]^{\sharp}}$, where
\[
u=\left\{
\begin{array}{cl}
x_k y_{\ell} & \mbox{if } {\bf c}^{(+)} = \ab_k \mbox{ and }
{\bf c}^{(-)} = \bb_\ell, \\
z y_\ell &  \mbox{if } {\bf c}^{(+)} = {\bf 0} \mbox{ and }
{\bf c}^{(-)} = \bb_\ell, \\
x_k z&  \mbox{if } {\bf c}^{(+)} = \ab_k \mbox{ and }
{\bf c}^{(-)} = {\bf 0}, \\
z^2 &  \mbox{if } {\bf c}^{(+)} = {\bf c}^{(-)} = {\bf 0}.
\end{array}
\right.
\]
If $z$ divides $u$, then ${\rm in}_{<_{\rm rev}}(f) = x_iy_j$,
where ${\rm in}_{<_{\rm rev}}(f)$ is the initial monomial of 
$f \in K[{\bf x}, {\bf y}, z]$.
If $z$ cannot divide $u$, then, since $\pi(x_k)$ divides $\pi(x_i)$,
one has $x_k <_A x_i$ and ${\rm in}_{<_{\rm rev}}(f) = x_iy_j$. 
Hence
\[
\{ \, x_i y_j \, : \, (i,j) \in {\mathcal E} \, \}
\subset
{\rm in}_{<_{\rm rev}}(I_{[-B, A]^{\sharp}}).
\]

Now, let %write
${\mathcal M}_{A}$ (resp. ${\mathcal M}_B$) be
the minimal set of squarefree monomial generators of % the initial ideal 
${\rm in}_{<_A} (I_{A^\sharp})$
(resp. ${\rm in}_{<_B} (I_{B^\sharp})$).
Suppose that % the initial ideal 
${\rm in}_{<_{\rm rev}}
( I_{[-B, A]^{\sharp}})$ cannot be generated by the set
of squarefree monomials
\[
{\mathcal M}=
\{ \, x_i y_j \, : \, (i,j) \in {\mathcal E} \, \}
\cup 
{\mathcal M}_A
\cup 
{\mathcal M}_B \, \, \,  
( \, \subset 
{\rm in}_{<_{\rm rev}}( I_{[-B, A]^{\sharp}}) \, ).
\]
The following fact (\cite[(0.1), p.~1914]{rootsystem}) on Gr\"obner bases 
is known:

\bigskip

A finite set ${\mathcal G}$ of $I_A$ is a Gr\"obner basis with respect to $<$
if and only if 
$\pi(u) \neq \pi(v)$ for any monomials
 $u \notin ( {\rm in}_<(g) : g \in {\mathcal G})$ and $v\notin( {\rm in}_<(g) : g \in {\mathcal G})$ with $u \neq v$.

\bigskip

\noindent
Since ${\mathcal G}$ with ${\mathcal M} = \{ {\rm in}_<(f) : f \in {\mathcal G}\}$
is not a Gr\"obner basis, it follows that
there exists a nonzero irreducible binomial 
$g = u - v$ belonging to $I_{[-B, A]^{\sharp}}$
such that each of $u$ and $v$
can be divided by none of the monomials belonging to 
${\mathcal M}$. 
Write
\[
u = \left(\,
\prod_{p \in P} x_{p}^{i_p} 
\right)
\left(\, 
\prod_{q \in Q} y_q^{j_q} 
\right),
\, \, \, \, \, \, 
v =
z^{\alpha} 
\left(\,
\prod_{r \in R} x_{r}^{k_r}
\right) 
\left(\, 
\prod_{s \in S} y_s^{\ell_s} 
\right),
\]
where $P$ and $R$ are subsets of $\{ 1, \ldots, n \}$, 
where $Q$ and $S$ are subsets of $\{ 1, \ldots, m \}$, 
where $\alpha$ is a nonnegative integer, 
and where each of $i_p, j_q, k_r, \ell_s$ is a positive integer.
Since $g = u -v$ is irreducible, one has
$P\cap R  = Q \cap S = \emptyset$.
Furthermore, the fact that each of $x_i y_j$ with $(i,j) \in {\mathcal E}$ 
can divide neither $u$ nor $v$ guarantees that
\[
\left(\,
\bigcup_{p \in P} {\rm supp} (\ab_p) 
\right)
\cap 
\left(\,
\bigcup_{q \in Q} {\rm supp} (\bb_q) 
\right)
=
\left(\,
\bigcup_{r \in R} {\rm supp} (\ab_r) 
\right)
\cap 
\left(\,
\bigcup_{s \in S} {\rm supp} (\bb_s) 
\right)
=
\emptyset.
\]
Since $\pi (u) = \pi (v)$, it follows that
\[
\sum_{p \in P} i_p \ab_p = \sum_{r \in R} k_r \ab_r,
\, \, \, \, \,  
\sum_{q \in Q} j_q \bb_q = \sum_{s \in S} \ell_s \bb_s.
\]
Let $\gamma_P = \sum_{p \in P} i_p$,
$\gamma_Q = \sum_{q \in Q} j_q$,
$\gamma_R = \sum_{r \in R} k_r$,
and $\gamma_S = \sum_{s \in S} \ell_s$.  Then
\[
\gamma_P + \gamma_Q = \alpha + \gamma_R + \gamma_S.
\]
Since $\alpha \geq 0$, it follows that
either $\gamma_P \geq \gamma_R$ or $\gamma_Q \geq \gamma_S$.
Let, say, $\gamma_P > \gamma_R$, then
\[
h = \prod_{p \in P} x_{p}^{i_p}
-
z^{\gamma_P - \gamma_R} \left(\, \prod_{r \in R} x_{r}^{k_r} \right)
\neq 0
\]
belongs to $I_{[-B, A]^{\sharp}}$
and ${\rm in}_{<_{\rm rev}}(h) = \prod_{p \in P} x_{p}^{i_p}$
divides $u$, a contradiction.  Hence $\gamma_P = \gamma_R$.
Then the binomial
\begin{eqnarray*}
\label{binomial}
h_{0} = \prod_{p \in P} x_{p}^{i_p} 
-
\prod_{r \in R} x_{r}^{k_r}
\end{eqnarray*}
belongs to $I_{[-B, A]^{\sharp}}$.  
If $h_{0} \neq 0$, then 
either
$\prod_{p \in P} x_{p}^{i_p}$ 
or
$\prod_{r \in R} x_{r}^{k_r}$
must belong to ${\rm in}_{<_{\rm rev}}(I_{[-B, A]^{\sharp}})$.
This contradicts the fact that
each of $u$ and $v$
can be divided by none of the monomials belonging to 
${\mathcal M}$.  Hence $h_{0} = 0$
and $P = R = \emptyset$.
Similarly, $Q = S = \emptyset$.  Hence $\alpha = 0$ and $g = 0$.
This contradiction guarantees that
${\mathcal M}$ is the minimal set of squarefree monomial generators
of ${\rm in}_{<_{\rm rev}} (I_{[-B, A]^{\sharp}})$, as desired.
\end{proof}

% A convex polytope ${\mathcal P} \subset \RR^{d}$ 
% is called {\em integral} if every vertex of $\Pc$ belongs to $\ZZ^{d}$.
Given an integral convex polytope $\Pc \subset \RR^{d}$, we write
$A_{\Pc}$ for the integer matrices whose column vectors are those
${\bf a} \in \ZZ^{d}$ belonging to $\Pc$.
The toric ring $K[A_{\Pc}]$ is often called the toric ring of $\Pc$.
A triangulation $\Delta$ of $\Pc$ with using the vertices belonging to
$\Pc \cap \ZZ^{d}$ is {\em unimodular} if the normalized volume
(\cite[p.~253]{dojoEN})
of each facet of $\Delta$ is equal to $1$ and is {\em flag} if
every minimal nonface of $\Delta$ is an edge.  
It follows from \cite[Chapter 8]{Sturmfels} that
if the toric ideal $I_{A_{\Pc}}$ of $A_{\Pc}$ possesses a squarefree initial 
ideal, then $\Pc$ possesses a unimodular triangulation.  Furthermore
if $I_{A_{\Pc}}$ possesses an initial ideal generated by quadratic
squarefree monomials, then $\Pc$ possesses a unimodular triangulation
which is flag.

An integral convex polytope $\Pc \subset \RR^{d}$ of dimension $d$
is called {\em Fano} if the origin of $\RR^{d}$ is a unique integer 
point belonging to the interior of $\Pc$.  We say that
a Fano polytope $\Pc$ is {\em Gorenstein} if the dual polytope $\Pc^{\vee}$ 
of $\Pc$ is again integral (\cite{Batyrev}, \cite{Hibidual}).
A {\em smooth} Fano polytope is a simplicial Fano polytope 
$\Pc \subset \RR^{d}$ for which the $d$ vertices of
each facet of $\Pc$ is a $\ZZ$-basis of $\ZZ^{d}$. 

\begin{Lemma}
\label{basiclemma}
Let $\Pc \subset \RR^{d}$ be an integral convex polytope
of dimension $d$ for which
${\bf 0} \in \ZZ^{d}$ belongs to $\Pc$.
Suppose that
there is a $d \times d$ minor $A'$ of $A_{\Pc}$ with
$\det(A') = \pm1$
and that
$I_{A_{\Pc}}$ possesses a squarefree initial ideal 
with respect to a reverse lexicographic order whose smallest variable 
corresponds to the column ${\bf 0} \in \ZZ^{d}$ of $A_{\Pc}$. 
Then, for each facet $\Fc$ of $\Pc$ with ${\bf 0} \not\in \Fc$,
one has $\ZZ \Fc =\ZZ^d$, where 
\[
\ZZ \Fc= \sum_{ \ab \, \in \, \Fc \, \cap \, \ZZ^{d} }\ZZ \ab, 
\] 
and the equation of the supporting hyperplane $\Hc \subset \RR^{d}$
with $\Fc \subset \Hc$ is of the form
\[
a_{1}z_{1} + \cdots + a_{d} z_{d} = 1
\]
with each $a_{j} \in \ZZ$.

In particular if $\Pc$ is a Fano polytope,
then $\Pc$ is Gorenstein.  
Furthermore if $\Pc$ is a simplicial
Fano polytope, then $\Pc$ is a smooth Fano polytope.
\end{Lemma}

\begin{proof}
Let $\Delta$ be the {\em pulling triangulation} (\cite[p.~268]{dojoEN}) 
coming from a squarefree initial ideal 
with respect to a reverse lexicographic order whose smallest variable 
corresponds to the column ${\bf 0} \in \ZZ^{d}$ of $A_{\Pc}$.
A crucial fact is that 
the origin of $\RR^{d}$ belongs to each facet of $\Delta$. 
Let $F$ be a facet of $\Delta$ with the vertices 
${\bf 0}, {\bf b}_{1}, \ldots, {\bf b}_{d}$
for which $\{ {\bf b}_{1}, \ldots, {\bf b}_{d} \} \subset \Fc$.
The existence of a $d \times d$ minor $A'$ of $A_{\Pc}$ with
$\det(A') = \pm1$ guarantees that the normalized volume of $F$ coincides
with $|\det(B)|$, where 
$B = [{\bf b}_{1}, \ldots, {\bf b}_{d}]$.
Since $\Delta$ is unimodular, one has $\det(B) = \pm1$.
Hence $\{ {\bf b}_{1}, \ldots, {\bf b}_{d} \}$ is a $\ZZ$-basis of $\ZZ^{d}$
and $\ZZ \Fc =\ZZ^d$ follows.  Moreover the hyperplane $\Hc \subset \RR^{d}$
with each ${\bf b}_{j} \in \Hc$ is of the form
$a_{1}z_{1} + \cdots + a_{d} z_{d} = 1$ with each $a_{j} \in \ZZ$, as desired.
\end{proof}

\begin{Corollary}
\label{Fano}
Work with the same situation as in Theorem \ref{squarefree}.
Let $\Pc \subset \RR^{d}$ be the integral convex polytope which is 
the convex hull of 
$\{ - {\bf b}_{1}, \ldots, - {\bf b}_{m}, 
{\bf a}_{1}, \ldots, {\bf a}_{n} \}$.
Suppose that ${\bf 0} \in \ZZ^{d}$ belongs to the interior of $\Pc$ % with 
% $\Pc \cap \ZZ^{d} = \{ {\bf 0}, 
% - {\bf b}_{1}, \ldots, - {\bf b}_{m}, 
% {\bf a}_{1}, \ldots, {\bf a}_{n} \}$ 
and that  
there is a $d \times d$ minor $A'$ of $A_{\Pc}$ with
$\det(A') = \pm1$. 
Then $\Pc$ is a Gorenstein Fano polytope. 
Furthermore if $\Pc$ is a simplicial
polytope, then $\Pc$ is a smooth Fano polytope.
\end{Corollary}

\begin{Example}
{\rm
Let $A_1$ and $A_2$ be the following matrices:
$$
A_1 = 
\begin{bmatrix}
1 & 0\\
0 & 1
\end{bmatrix},\ \ 
A_2= 
\begin{bmatrix}
1 & 0 & 1\\
0 & 1 & 1
\end{bmatrix}.
$$
Then, %it is easy to check that
$A_i$ and $A_j$ are of harmony and 
satisfy the condition in Theorem \ref{squarefree}
for any $1 \leq i \le j \leq 2$.
By Corollary \ref{Fano}, we have three Gorenstein Fano polygons.
It is known that there are exactly 16 Gorenstein Fano polygons
(\cite[p.382]{CLS}).
}
\end{Example}

\section{Convex polytopes arising from simplicial complexes}
Let $[d] = \{1, \ldots, d \}$ and
${\bf e}_1, \ldots, \eb_{d}$ the standard coordinate unit vectors 
of ${\RR}^d$.
Given a subset $W \subset [d]$, one has
$\rho(W) = \sum_{j \in W} {\eb}_j \in {\RR}^d$.
In particular $\rho(\emptyset)$ is the origin of $\RR^d$.
Let $\Delta$ be a simplicial complex on the vertex set 
$[d]$.  Thus $\Delta$ is a collection of subsets of $[d]$ with 
$\{ i \} \in \Delta$ for each $i \in [d]$ such that if
$F \in \Delta$ and $F' \subset F$, then $F' \in \Delta$.
In particular $\emptyset \in \Delta$.
The {\em incidence matrix} $A_\Delta$ of $\Delta$ is the matrix
whose columns are those $\rho(F)$ with $F \in \Delta$.
We write $\Pc_\Delta \subset \RR^{d}$ for the $(0, 1)$-polytope
which is the convex hull of
$\{ \, \rho(F) \; : \; F \in \Delta \, \}$ in $\RR^d$.
One has $\dim \Pc_\Delta = d$.
It follows from the definition of simplicial complexes that 

\begin{Lemma}
\label{obvious}
Let $\Delta$ and $\Delta'$ be simplicial complexes on $[d]$.
Then $A_\Delta$ and $A_{\Delta'}$ are of harmony.
\end{Lemma}

Following Lemma \ref{obvious}
together with 
Theorem \ref{squarefree}, it is reasonable 
to study the problem when
the toric ideal $I_{A_\Delta}$ of a simplicial complex $\Delta$ possesses 
a squarefree initial ideal with respect to a reverse lexicographic order
whose smallest variable corresponds to 
the column ${\bf 0} \in \ZZ^{d}$ of $A_\Delta$.

Let $\Delta$ be a simplicial complex on $[d]$.
% and
% $A_{\Delta} = [{\bf 0}, \ab_{1}, \ldots, \ab_{n}]$.
Since $\{ i \} \in \Delta$ for each $i \in [d]$,
% it follows that $\ZZ \ab_{1} + \cdots + \ZZ \ab_{n} = \ZZ^{d}$.
the $d \times d$ identity matrix is a $d \times d$ minor of $A_\Delta$.
It then follows from Lemma \ref{basiclemma} that 

\begin{Corollary}
\label{COR}
Let $\Delta$ be a simplicial complex on $[d]$.
Suppose that $I_{A_{\Delta}}$ possesses 
a squarefree initial ideal with respect to a reverse lexicographic order
whose smallest variable corresponds to 
the column ${\bf 0} \in \ZZ^{d}$ of $A_{\Delta}$.
Then, for each facet $\Fc$ of $\Pc_{\Delta}$ with ${\bf 0} \not\in \Fc$,
one has $\ZZ \Fc =\ZZ^d$
and the equation of the supporting hyperplane $\Hc \subset \RR^{d}$
with $\Fc \subset \Hc$ is of the form
$a_{1}z_{1} + \cdots + a_{d} z_{d} = 1$
with each $a_{j} \in \ZZ$.
\end{Corollary}

Let $G$ be a finite simple graph on $[d]$
and $E(G)$ the set of edges of $G$.
(Recall that a finite graph is simple if $G$ possesses no loop and no multiple edge.)
A subset $W \subset [d]$ is called {\em stable}  
if, for all $i$ and $j$ belonging to $W$ with $i \neq j$,
one has $\{i,j\} \not\in E(G)$.
Let $S(G)$ denote the set of stable sets of $G$.
One has $\emptyset \in S(G)$ and $\{ i \} \in S(G)$
for each $i \in [d]$.
Clearly $S(G)$ is a simplicial complex on $[d]$.
The {\em stable set polytope} $\Qc_G \subset \RR^{d}$
of $G$ is the $(0, 1)$-polytope $\Pc_{S(G)} \subset \RR^d$
arising from the simplicial complex $S(G)$.
A finite simple graph is said to be {\em perfect} 
(\cite{sptheorem}) if, 
for any induced subgraph $H$ of $G$
including $G$ itself,
the chromatic number of $H$ is
equal to the maximal cardinality of cliques 
of $H$.
(A chromatic number of $G$ is the smallest integer $t$ for which
there exist stable set $W_{1}, \ldots, W_{t}$ of $G$ with
$[d] = W_{1} \cup \cdots \cup W_{t}$
and a clique of $G$ is a subset $W \subset [d]$
which is a stable set of the complementary graph $\overline{G}$ of $G$.)
A complementary graph of a perfect graph is perfect (\cite{sptheorem}).

Recall that an integer matrix $A$
is {\em compressed} (\cite{compressed}, \cite{Sul}) if the initial ideal 
of the toric ideal $I_{A}$ is squarefree with respect to 
any reverse lexicographic order.  

\begin{Example}
\label{EXperfect}
{\em
Let $G$ be a perfect graph on $[d]$. 
Then $A_{\Delta}$, where $\Delta = S(G)$,
is compressed (\cite[Example 1.3 (c)]{compressed}).
Let $G$ and $G'$ be perfect graphs on $[d]$
and $\Qc \subset \RR^{d}$ be the Fano polytope 
which is the convex hull of $\Qc_{G} \cup (- \Qc_{G'})$.
It then follows from Corollary \ref{Fano} together with Lemma \ref{obvious}
that $\Qc$ is Gorenstein.
}
\end{Example}

\begin{Lemma}
\label{hole}
Let $\Delta$ be one of the following simplicial complexes:
\begin{enumerate}
\item[\rm (i)]
the simplicial complex on $[e]$ with the facets
$[e] \setminus \{ i \}$, $1 \leq i \leq e$, where $e \geq 3$;
\item[\rm (ii)]
$S(G)$, where $G$ is an odd hole of length $2\ell + 1$, where $\ell \geq 2$;
\item[\rm (iii)]
$S(G)$, where $G$ is an odd antihole of length $2\ell + 1$, where $\ell \geq 2$;
\end{enumerate}
Let $<$ be any reverse lexicographic order
whose smallest variable corresponds to 
the column ${\bf 0}$ of $A_\Delta$. 
Then the initial ideal ${\rm in}_{<}(I_{A_{\Delta}})$ cannot be squarefree.
(Recall that an odd hole is an induced odd cycle of length $\geq 5$
and an odd antihole is the complementary graph of an odd hole.) 
\end{Lemma}

\begin{proof}
By virtue of Corollary \ref{COR},
we find a supporting hyperplane $\Hc$
of $\Pc_{\Delta}$ with ${\bf 0} \not\in \Hc$ for which 
$\Hc \cap \Pc_{\Delta}$ is a facet of $\Pc_{\Delta}$
such that the equation of $\Hc$ cannot be of the form
$a_{1}z_{1} + \cdots + a_{d} z_{d} = 1$
with each $a_{j} \in \ZZ$.
In each of (i), (ii) and (iii), 
the equation of a desired hyperplane $\Hc$
is as follows:
\begin{enumerate}
\item[(i)]
$\sum_{i=1}^{e} z_{i} = e - 1$;
\item[(ii)]
$\sum_{i=1}^{2\ell+1} z_{i} = \ell$;
\item[(iii)]
$\sum_{i=1}^{2\ell+1} z_{i} = 2$.
\end{enumerate}
In (i), it is easy to see that $\Hc \cap \Pc_{\Delta}$ is a facet of $\Pc_{\Delta}$.  
In each of (ii) and (iii),
it is known (\cite{Padberg}, \cite{Wagler}) that
$\Hc \cap \Pc_{\Delta}$ is a facet of $\Pc_{\Delta}$.  
\end{proof}

Let $B =[{\bf b}_1,\ldots, {\bf b}_m] \in \ZZ^{d \times m}$ be
a submatrix of $A =[{\bf a}_1,\ldots, {\bf a}_n] \in \ZZ^{d \times n}$.
Then, $K[B]$ is called a {\em combinatorial pure subring} of $K[A]$
if the convex hull of $\{{\bf b}_1,\ldots, {\bf b}_m\}$ is a face of 
the convex hull of  $\{{\bf a}_1,\ldots, {\bf a}_n\}$.
For any combinatorial pure subring $K[B]$ of $K[A]$,
it is known that, if the initial ideal of $I_A$ is squarefree,
then so is the corresponding initial ideal of $I_B$.
See \cite{OHHcp, OHS} for details.

\begin{Lemma}
\label{FACET}
Let $\Delta$ be a simplicial complex on $[d]$ and $\Delta'$ an induced 
subcomplex of $\Delta$ which is one of (i), (ii) and (iii) of Lemma {\rm \ref{hole}}. 
Let $<$ be any reverse lexicographic order
whose smallest variable corresponds to 
the column ${\bf 0} \in \ZZ^{d}$ of $A_\Delta$. 
Then the initial ideal ${\rm in}_{<}(I_{A_{\Delta}})$ cannot be squarefree.
\end{Lemma}

\begin{proof}
Let $\Delta'$ be the induced subcomplex of $\Delta$ on $V$, 
where $V \subset [d]$, and
$<'$ the reverse lexicographic order 
induced by $<$.
Lemma \ref{hole} says that ${\rm in}_{<'}(I_{A_{\Delta'}})$ cannot be squarefree.
Since $\Delta'$ is an induced subcomplex of $\Delta$,
it follows that $\Pc_{\Delta'}$ is a face of $\Pc_{\Delta}$.
Thus $K[A_{\Delta'}]$ is a combinatorial pure subring of $K[A_{\Delta}]$
and hence
${\rm in}_{<}(I_{A_{\Delta}})$ cannot be squarefree, as required.
\end{proof}

We are now in the position to state a combinatorial 
characterization of simplicial
complexes $\Delta$ on $[d]$
for which the toric ideal $I_{A_{\Delta}}$ possesses 
a squarefree initial ideal with respect to a reverse lexicographic order
whose smallest variable corresponds to 
the column ${\bf 0} \in \ZZ^{d}$ of $A_\Delta$.

\begin{Theorem}
\label{compressed}
Let $\Delta$ be a simplicial complex on $[d]$.
Then the following conditions are equivalent:
\begin{enumerate}
\item[{\rm (i)}]
There exists a perfect graph $G$ on $[d]$ with
$\Delta = S(G)$;
\item[{\rm (ii)}]
$A_\Delta$ is compressed;
\item[{\rm (iii)}]
$I_{A_\Delta}$ possesses 
a squarefree initial ideal 
with respect to a reverse lexicographic order
whose smallest variable corresponds to 
the column ${\bf 0} \in \ZZ^{d}$ of $A_\Delta$.
\end{enumerate}
\end{Theorem}

\begin{proof}
In \cite[Example 1.3 (c)]{compressed},
(i) $\Rightarrow$ (ii) is proved. 
(See also \cite[\S 4]{GPT}.)
Moreover,
(ii) $\Rightarrow$ (iii) is trivial.
Now, in order to show (iii) $\Rightarrow$ (i),
we fix a reverse lexicographic
order $<$ whose smallest variable corresponds ${\bf 0} \in \ZZ^{d}$ 
of $A_\Delta$. 

\smallskip

\noindent
{\bf (First Step)}  Suppose that
there is {\em no} finite simple graph $G$ on $[d]$ with
$\Delta = S(G)$.
Given a simplicial complex $\Delta$ on $[d]$, 
there is a finite simple
graph $G$ on $[d]$ with $\Delta = S(G)$ if and only if $\Delta$ is flag,
i.e, every minimal nonface 
of $\Delta$ is an edge of $\Delta$.
(See, e.g., \cite[Lemma 9.1.3]{HerHibi}.
Note that $S(G)$ is the clique complex of the complement graph of $G$.)
Let $\Delta$ be a simplicial complex which is not flag
and $V \subset [d]$, where $|V| \geq 3$,
a minimal nonface of $\Delta$.  One has 
$V \setminus \{ i \} \in \Delta$ for all $i \in V$.
Thus the induced subcomplex $\Delta'$ of $\Delta$ on V coincides with 
the simplicial complex (i) of Lemma \ref{hole}.   

\smallskip

\noindent
{\bf (Second Step)} Let $G$ be a nonperfect graph on $[d]$ with
$A_\Delta = S(G)$.
The strong perfect graph theorem \cite{sptheorem}
guarantees that $G$ 
possesses either an odd hole or an odd antihole.
Thus $\Delta$ contains an induced subcomplex 
$\Delta'$ which coincides with either
(ii) or (iii) of Lemma \ref{hole}.

\smallskip

As a result, Lemma \ref{FACET} says that 
$I_{A_\Delta}$ possesses no squarefree initial ideal 
with respect to a reverse lexicographic order
whose smallest variable corresponds to 
the column ${\bf 0} \in \ZZ^{d}$ of $A_\Delta$.
This completes the proof of (iii) $\Rightarrow$ (i).
\, \, \, \, \, \, \, \, \, \, \,
\end{proof}

\begin{Example}
{\em
Let $A \in \ZZ^{d \times n}$ for which each entry of $A$ belongs
to $\{ 0, 1 \}$ and $I_{A^{\sharp}}$ the toric ideal of $A^{\sharp}$.
In general, even if $I_{A^{\sharp}}$ possesses 
a squarefree initial ideal with respect to a reverse lexicographic order
whose smallest variable corresponds to 
the column ${\bf 0} \in \ZZ^{d}$ of $A^{\sharp}$,
the matrix $A^{\sharp}$ may not be compressed.  For example, if
\begin{eqnarray*}
A = \left[ 
\begin{array}{ccccccc}
1& 0& 1& 1& 0& 0&0\\
1& 1& 0& 0& 0& 0&0\\
0& 1& 1& 0& 0& 0&0\\
0& 0& 0& 1& 1& 0&1\\
0& 0& 0& 0& 1& 1&0\\
0& 0& 0& 0& 0& 1&1\\
\end{array} 
\right],
\end{eqnarray*}
then $I_{A^{\sharp}}$ is generated by 
$x_{1}x_{3}x_{5}x_{7}- x_{2}x_{4}^{2}x_{6}$.  Thus
the initial ideals of $I_{A^{\sharp}}$ with respect to the reverse 
lexicographic order induced by the ordering 
\[
z < x_{2} < x_{1} < x_{3} < x_{4} < x_{5} < x_{6} < x_{7}
\]
is generated by $x_{1}x_{3}x_{5}x_{7}$, while the initial ideals 
of $I_{A^{\sharp}}$ with respect to 
the reverse lexicographic order induced by the ordering
\[
z < x_{1} < x_{2} < x_{3} < x_{4} < x_{5} < x_{6} < x_{7}
\]
is generated by $x_{2}x_{4}^{2}x_{6}$.  
Even though $A^{\sharp}$ satisfies the condition (iii) of
Theorem \ref{compressed}, the integer matrix $A^{\sharp}$ 
cannot be compressed.
}
\end{Example}

Apart from Theorem \ref{compressed}, we can ask the problem when
the convex polytope $\Pc \subset \RR^{d}$ which is the convex hull
of $\Pc_{\Delta} \cup (- \Pc_{\Delta'})$, where $\Delta$ and $\Delta'$
are simplicial complexes on $[d]$, is a Gorenstein Fano polytope.

\begin{Theorem}
\label{GORFANO}
Let $\Delta$ and $\Delta'$ be simplicial complexes on $[d]$
and $\Pc \subset \RR^{d}$ the convex polytope  which is the convex hull
of $\Pc_{\Delta} \cup ( - \Pc_{\Delta'})$.
Then $\Pc$ is a Gorenstein Fano polytope
if and only if there exist perfect graphs $G$ and $G'$ on $[d]$
with $\Delta = S(G)$ and $\Delta' = S(G')$.
\end{Theorem}

\begin{proof}
The ``If'' part follows from Example \ref{EXperfect}.
To see why the ``Only If'' part is true,  
suppose that either
$\Delta$ is not flag
or there is a nonperfect graph $G$ with $\Delta = S(G)$.
Since $\Pc \subset \RR^{d}$ is a Gorenstein Fano polytope,
the equation of the supporting hyperplane $\Hc \subset \RR^{d}$
for which $\Hc \cap \Pc$ is a facet of $\Pc$ is of the form
$a_{1}z_{1} + \cdots + a_{d}z_{d} = 1$ with each $a_{j} \in \ZZ$.

Let $\Delta$ be not flag and $V \subset [d]$ with $|V| \geq 3$
for which $V \setminus \{ i \} \in \Delta$ for all $i \in V$
and $V \notin \Delta$.
Let, say, $V = [e]$ with $e \geq 3$.
Then the hyperplane $\Hc' \subset \RR^{d}$ 
defined by the equation
$z_{1} + \cdots + z_{e} = e - 1$ 
is a supporting hyperplane of $\Pc$.  Let $\Fc$ be a facet of $\Pc$
with $\Hc' \cap \Pc \subset \Fc$ and  
$a_{1}z_{1} + \cdots + a_{d}z_{d} = 1$ with each $a_{j} \in \ZZ$
the equation of the supporting hyperplane $\Hc \subset \RR^{d}$
with $\Fc \subset \Hc$.
Since $\rho(V \setminus \{ i \}) \in \Hc$ for all $i \in V$,
one has $\sum_{j \in [e] \setminus \{ i \}} a_{j} = 1$.
Thus $(e - 1) (a_{1} + \cdots + a_{e}) = e$.
Hence $a_{1} + \cdots + a_{e} \not\in \ZZ$, a contradiction.

Let $\Delta = S(G)$, where $G$ possesses an odd hole $C$ 
of length $2\ell + 1$ with the vertices, say, $1, \ldots, 2\ell + 1$,
where $\ell \geq 2$.
Then the hyperplane $\Hc' \subset \RR^{d}$ 
defined by the equation
$z_{1} + \cdots + z_{2\ell + 1} = \ell$ 
is a supporting hyperplane of $\Pc$.  Let $\Fc$ be a facet of $\Pc$
with $\Hc' \cap \Pc \subset \Fc$ and  
$a_{1}z_{1} + \cdots + a_{d}z_{d} = 1$ with each $a_{j} \in \ZZ$
the equation of the supporting hyperplane $\Hc \subset \RR^{d}$
with $\Fc \subset \Hc$.
The maximal stable sets of $C$ is 
\[
\{1, 3, \ldots, 2\ell - 1\}, \{2, 4, \ldots, 2\ell \}, \ldots, 
\{2\ell + 1, 2, 4, \ldots, 2\ell - 2\}
\]
and each $i \in [2\ell - 2]$ appears $\ell$ times in the above list.
Since, for each maximal stable set $U$ of $C$, 
one has $\sum_{i \in U} a_{i} = 1$, it follows that 
$\ell(a_{1} + \cdots + a_{2\ell + 1}) = 2\ell + 1$.
Hence $a_{1} + \cdots + a_{e} \not\in \ZZ$, a contradiction.

Let $\Delta = S(G)$, where $G$ possesses an odd antihole $C$
with the vertices, say, $1, \ldots, 2\ell + 1$, where $\ell \geq 2$.
Then the hyperplane $\Hc' \subset \RR^{d}$ 
defined by the equation
$z_{1} + \cdots + z_{2\ell + 1} = 2$ 
is a supporting hyperplane of $\Pc$.  Let $\Fc$ be a facet of $\Pc$
with $\Hc' \cap \Pc \subset \Fc$ and  
$a_{1}z_{1} + \cdots + a_{d}z_{d} = 1$ with each $a_{j} \in \ZZ$
the equation of the supporting hyperplane $\Hc \subset \RR^{d}$
with $\Fc \subset \Hc$.
The maximal stable sets of $C$ is
\[
\{1, 2\}, \{2, 3\}, \ldots, \{2\ell+1, 1\}
\]
and each $i \in [2\ell - 2]$ appears twice in the above list.
Since, for each maximal stable set $U$ of $C$, 
one has $\sum_{i \in U} a_{i} = 1$, it follows that 
$2(a_{1} + \cdots + a_{2\ell + 1}) = 2\ell + 1$.
Hence $a_{1} + \cdots + a_{e} \not\in \ZZ$, a contradiction.
\, \, \, \, \, \, \, \, \, \, 
\, \, \, \, \, \, \, \, \, \, 
\, \, \, \, \, \, \, \, \, \, \,
\end{proof}

\bigskip

\noindent
{\bf Acknowledgment.}
The authors are grateful to an anonymous referee for 
useful suggestions and helpful comments.

\end{document}